\documentclass{amsart}
\usepackage{tikz-cd}

\usepackage[all]{xy}
\usepackage{enumitem}

\newtheorem{theorem}{Theorem}[section]
\newtheorem{lemma}[theorem]{Lemma}

\theoremstyle{definition}

\newtheorem{example}[theorem]{Example}

\theoremstyle{remark}

\numberwithin{equation}{section}

\newcommand{\cartesian}[3]{\ensuremath{#1_{#2} \times \cdots \times #1_{#3}}}

\newcommand{\Tensor}[4]{\ensuremath{#1_{#2}}{\otimes}_{#4}\cdots{\otimes}_{#4}#1_{#3}} \newcommand{\spaces}[3]{\ensuremath{#1_{#2}, \ldots , #1_{#3}}}
\begin{document}

\title[Lipschitz p-summability  of multilinear operators]{\textbf Lipschitz p-summing  multilinear operators correspond to Lipschitz p-summing operators}

\author{Maite Fern\'{a}ndez-Unzueta}
\address{Centro de Investigaci\'{o}n en Matem\'{a}ticas (Cimat), A.P. 402 Guanajuato, Gto., C:P. 36000 M\'{e}xico}
\curraddr{}
\email{maite@cimat.mx}
\thanks{The author was partially supported by  CONACyT project 284110}

\date{}

\dedicatory{}

\commby{Stephen Dilworth}

\begin{abstract}

We give conditions  that ensure that an operator  satisfying  a
 Piestch domination in a given  setting  also satisfies  a Piestch domination in  a different setting.
 From this we derive  that
 a bounded mutlilinear operator $T$ is Lipschitz $p$-summing  if and only if the  mapping  $ f_T(x_1\otimes\cdots \otimes x_n):=T(x_1,\ldots, x_n)$ is Lipschitz $p$-summing.
 The results are based on   the  projective tensor norm. An example with  the Hilbert tensor norm is provided to show that the statement  may not hold when  a reasonable cross-norm other than the projective tensor norm is considered.
\end{abstract}

\maketitle

\section{Introduction}\label{sec: intro and pre}

 A  bounded multilinear operator $T$ on the tensor product of Banach spaces uniquely determines  a Lipschitz mapping $f_T$  by the relation $f_T(x_1\otimes\cdots \otimes x_n):=T(x_1,\ldots,x_n)$, as explained in  \cite{MFU}.  This relation makes it possible to use the theory of Lipschitz mappings  to study multilinear operators on Banach spaces.

    In \cite{FJ09} J. Farmer and W.B. Johnson introduced the class of  Lipschitz $p$-summing mappings   defined on metric spaces. The above-mentioned relation between  $T$ and $f_T$ makes it natural to ask  for the multilinear mappings $T$ such that $f_T$ is a Lipschitz $p$-summing mapping.

    In this paper we prove that such multilinear mappings   are precisely  the class of Lipschitz $p$-summing multilinear operators introduced  in \cite{Ang-FU}.  The corresponding  Lipschitz $p$-summing norms satisfy {{}{$\pi_p^L(f_T)= \pi^{Lip}_{p}(T)$ }}  (Theorem \ref{thm: Lip=Lip}).
     This result generalizes  the  case of linear operators  proved in  \cite[Theorem 2]{FJ09}.
      A main  part of its  proof consists   of moving  from  a   Pietsch-type  domination  of $f_T$ (as a Lipschitz operator) to another Pietsch-type domination of $f_T$ (as a $\Sigma$-operator associated to the multilinear mapping $T$).
      Theorem \ref{thm: Induced Pietsch} provides conditions in a more general setting, to guarantee that such a motion is possible.

     We  use  Hilbert-Schmidt multilinear operators  to show  that  Theorem \ref{thm: Lip=Lip} may not  hold when a   reasonable cross-norm  other  than the projective tensor norm is considered.

\vspace{.5cm}

Throughout the paper $X_1,\ldots, X_n$ and $Y$ are  real Banach spaces and  $B_{X_i}$ is  the closed unit ball of a space   $X_i$.  The completed projective tensor product of $X_1,\ldots,X_n$
is denoted by  $\hat{\otimes}_{\pi}{X_i}$ and the space  of multilinear bounded operators from $\cartesian{X}{1}{n}$ to $Y$, by   $ \mathcal{L}(\spaces{X}{1}{n}; Y)$. The linear operator determined by $T$ with  domain  the tensor product  will be denoted by $\hat{T}$.
The linear theory of Banach spaces that we will  use  can be found in
\cite{DJT}, the theory of tensor norms in \cite{DeFlo} and \cite{Ryan-libro} and the Lipchitz theory in \cite{BenyLind} and \cite{Weaver-libro}.

  Each bounded  multilinear operator $T\in\mathcal{L}(X_1,\ldots,X_n;Y)$ can be uniquely  associated with a  Lipschitz mapping $f_T:\Sigma_{X_1,\ldots,X_n}\rightarrow Y$  where $\Sigma_{X_1,\ldots,X_n}:=\{x_1\otimes\cdots\otimes x_n; x_i\in X_i\}$ is endowed with the   metric induced by   $ X_1\hat{\otimes}_\pi \cdots \hat{\otimes}_\pi X_n$ \cite[Theorem 3.2]{MFU}. The mapping $f_T(\Tensor{x}{1}{n}{}):=T(\spaces{x}{1}{n})$ and $\Sigma_{X_1,\ldots,X_n}$ are called the $\Sigma$-operator associated to $T$ and  the Segre cone of $\spaces{X}{1}{n}$, respectively.

   A norm $\gamma$ on the vector space $\Tensor{X}{1}{n}{}$ is said to be a { reasonable cross-norm} if it has the following  two properties:
   ($i$) $\gamma(x_1\otimes\dots\otimes x_n)\leq \|x_1\|\cdots \|x_n\|$ for every $x_i\in X_i, \; i=1,\ldots n$, and ($ii$)
   For every $x_i^*\in X_i^*$, the linear functional $x_1^*\otimes\dots\otimes x_n^*$ on $\Tensor{X}{1}{n}{}$ is bounded, and $\|x_1^* \otimes\dots\otimes x_n^*\|\leq \|x_1^*\|\cdots \|x_n^*\|$. We denote the completed space as $\hat{\otimes}_{\gamma}X_i$. The space of  multilinear operators such that its  associated linear operator $\hat{T}$ is continuous  on $\hat{\otimes}_{\gamma}X_i$,  will be denoted  ${\mathcal{L}_{\gamma}\left(\spaces{X}{1}{n}; Y \right)}$ and ${\mathcal{L}_{\gamma}\left(\spaces{X}{1}{n}\right)}$ when $Y$ is the scalar field.

   For a fixed reasonable cross-norm $\gamma$ and   $1 \leq p < \infty$ we say,  as in
      \cite[Definition 5.1]{Ang-FU},
  that  $T\in {\mathcal{L}_{\gamma}\left(\spaces{X}{1}{n}; Y \right)}$    is Lipschitz $p$-summing with respect to $\gamma$  (briefly, $\gamma$-Lipschitz $p$-summing) if there exists $c>0$ such that for $k\in\mathbb{N}$, $i=1,\ldots,k$ and  every  $u_i:=(u_i^1,\ldots, u_i^n), v_i :=(v_i^1,\ldots, v_i^n)\in X_1\times\cdots \times X_n$,
 \begin{equation} \label{eq: def sigma p sumante}
  \sum_{i=1}^k \left\|{T}\left(u_i\right)-T \left(v_i\right)\right\|^p
  \leq    c^p\cdot \sup_{ \varphi \in B_{\mathcal{L}_{\gamma}\left(\spaces{X}{1}{n}\right)}} \sum_{i=1}^k \left|{{\varphi}}\left(u_i\right)-{\varphi}\left(v_i\right)\right|^p \end{equation}
 The best $c$ above is denoted  $\pi^{Lip,\gamma}_p(T)$.
  Whenever   $\gamma$ is the projective tensor norm, it will be omitted in the notation.

If $T$ is a linear operator (i.e., $n=1$ and  $\gamma$  is the norm of $X_1$)  it holds  that  $\pi^{Lip,\gamma}_p(T)$ is the usual $p$-summing norm of $T$ \cite[p.31]{DJT}.

\section{Main results}

 As introduced in  \cite{FJ09},    the Lipschitz $p$-summing  norm ($1\leq p <\infty$) $\pi_p^L(T)$ of a (possibly nonlinear) mapping $T:X\rightarrow Y$  between metric spaces  is the smallest constant $C$ so that for all $(x_i)_i$, $(y_i)_i$ in $X$ \begin{equation}\label{eq: FarJ}
 \sum d(T(x_i),T(y_i))^p\leq C^p \sup_{f\in B_{X^{\#}}} \sum |f(x_i)-f(y_i)|^p
  \end{equation}
where $d$ is the distance in $Y$ and    ${X^{\#}}$ is the  Lipschitz dual of $X$, that is,  the  space  of all real valued Lipschitz functions defined on $X$ that vanish at a specified   point $0\in X$. It is a Banach space with the Lipschitz norm and its unit ball  $B_{X^{\#}}$ is a compact Hausdorff space in the topology of pointwise convergence on $X$. In the case of vector valued functions, that is, when
$Y$ is a Banach space, we will use the notation  $Lip_0(X, Y )$ to designate  the Banach space of Lipschitz
functions $f :X \rightarrow Y $ such that $f(0) = 0$  with pointwise addition and the Lipschitz norm.

\begin{theorem}\label{thm: Lip=Lip} A multilinear operator  $T\in \mathcal{L}(\spaces{X}{1}{n};Y)$ is Lipschitz $p$-summing if and only if its associated $\Sigma$-operator $f_T:\Sigma_{X_1,\ldots, X_n}\rightarrow Y$ satisfies (\ref{eq: FarJ}) for some  $C>0$. In this case {{}{$\pi_p^L(f_T)= \pi^{Lip}_{p}(T)$. }}
\end{theorem}

  The  notions of Lipschitz $p$-summability   compared in this note   are two among  several existing  generalizations of the linear absolutely $p$-summing operators. Each of them comes with  a  generalization  of the so-called Pietsch domination theorem. We  will call  them  Pietsch-tpye dominations.   To the best of our knowledge,  all of them  are proved  adapting  the original linear argument, namely,  as an application of the Hanh-Banach separation theorem (an abstract formulation of $p$-summability can be found in  \cite{BotPelRue2}).  We refer to \cite[2.12]{DJT} to see a  proof of the original linear case.

For a compact Hausdorff  set $K$ let $\mathcal{C}(K)$  denote the space of continuous functions  on $K$ with the sup norm and $\mathcal{M}(K)$ the space of  measures with the variation norm identified with the dual space $\mathcal{C}(K)^*$ by the Riesz' representation theorem.  The next lemma summarizes some  general results in a form  that  will be used later.

\begin{lemma}\label{lem: rho1 rho2 probabilities}
 Let $\varphi:K_1\rightarrow K_2$ be  a continuous  mapping between compact (Hausdorff) spaces. Consider
    $\Phi:\mathcal{C}(K_2)\rightarrow \mathcal{C}(K_1)$    defined as  $\Phi(g):=g\circ \varphi$ and its adjoint operator $\Phi^*:\mathcal{M}(K_1)\rightarrow \mathcal{M}(K_2)$.
    Then, for  any $\mu\in\mathcal{M}(K_1)$  and every $g\in \mathcal{C}(K_2)$ it holds
\begin{equation}\label{eq: equal integrals}
 \int_{K_1} (g\circ \varphi)d\mu=\int_{K_2}gd(\Phi^*\mu).   \end{equation}
 If $\mu$ is a  probability measure on $K_1$, then $\Phi^*(\mu)$ is a probability measure on $K_2$. If $\mu \in \mathcal{M}(K_1)$  is such that $\Phi^*(\mu)$  is positive, then
 for every $h\in \mathcal{C}(K_2)$, $h\geq 0$ it holds
 $ 0\leq \int_{K_1} h\circ \varphi d\mu.$
\end{lemma}
\begin{proof}
All the facts used in this proof can be found in \cite{Semadeni65}. The mapping  $\Phi$ and its adjoint operator $\Phi^*$ are   bounded linear operators with $\|\Phi\|=\|\Phi^*\|=1$. The duality $\mathcal{M}(K_2)=\mathcal{C}(K_2)^*$  means  that
for  any $\mu\in\mathcal{M}(K_1)$  and every $g\in \mathcal{C}(K_2)$ (\ref{eq: equal integrals}) holds.

Assume that $\mu$ is a probability measure. For each  positive
$g\in \mathcal{C}(K_2)$, $g\circ\varphi \in \mathcal{C}(K_1)$ is positive.   Then
$0\leq  \int_{K_1} g\circ\varphi d\mu$. By (\ref{eq: equal integrals}),   the measure $\Phi^*(\mu)$ is positive, too.  Even more, applying (\ref{eq: equal integrals}) to the  constant  function  $g(k_2)\equiv 1 $ we get $\Phi^*(\mu)(K_2)=\mu(K_1)=1$. Consequently,
$\Phi^*(\mu)$  is   a probability measure.

Now, assume that $\Phi^*(\mu)$ is a positive measure.  Using
    (\ref{eq: equal integrals}) again,  it follows that for
    $h, f \in \mathcal{C}(K_2)$,  $0\leq h\leq f$
   $$ 0\leq  \int_{K_2} h d\Phi^*(\mu)= \int_{K_1} h\circ \varphi d\mu\leq \int_{K_2}f d\Phi^*(\mu)= \int_{K_1} f \circ \varphi d\mu.$$

\end{proof}

Our proof of Theorem \ref{thm: Lip=Lip} relies in  the possibility   of moving  from the  Pietsch-type  domination  of an operator in one context to the  Pietsch-type  domination of it  in other context. The following theorem provides conditions for this.
\begin{theorem}\label{thm: Induced Pietsch}
Fix $1\leq p< \infty$. Let $\varphi:K_1\rightarrow K_2$ be  a continuous  mapping between compact (Hausdorff) spaces,  a non-empty set   $   Z\subset \mathcal{C}(K_2)$,   $F:Z \rightarrow Y$   a mapping into a Banach space $Y$ and $C>0$. If   there exists
  a probability measure  $\rho_1$  on $K_1$   such that for every  $h\in Z$
\begin{equation}\label{eq: abstract Pietsc mu1}
 \|F(h)\|^p \leq  C^p \int_{K_1} |h\circ \varphi |^pd\rho_1
 \end{equation}
 then,  there exists a probability measure  $\rho_2$ in $K_2$, such that  for every  $h\in Z$
   \begin{equation}\label{eq: abstract Pietsc mu2}
\|F(h)\|^p \leq C^p  \int_{K_2} |h|^pd\rho_2.
 \end{equation}
If the continuous mapping $\varphi$ is surjective, then if $\rho_2$ is a probability measure on $K_2$ such that (\ref{eq: abstract Pietsc mu2}) holds, then there exists a probability measure $\rho_1$  on $K_1$ such that (\ref{eq: abstract Pietsc mu1}) holds.

\end{theorem}
\begin{proof}
  Let   $\rho_1$ be  a probability measure satisfying  (\ref{eq: abstract Pietsc mu1}). By Lemma \ref{lem: rho1 rho2 probabilities} we know that $\rho_2:=\Phi^*(\rho_1)$ is a probability measure on $K_2$ and that they satisfy (\ref{eq: equal integrals}).  It remains  to  check that  (\ref{eq: abstract Pietsc mu2}) holds. For  an $h\in Z$ consider  $g(k_2)=|h(k_2)|^p$. By (\ref{eq: equal integrals}) we have that
  $$\int_{K_1} |h \circ \varphi|^pd\rho_1=\int_{K_2}|h|^pd\rho_2.   $$
   From this it is immediate that whenever  $F$ satisfies (\ref{eq: abstract Pietsc mu1}), it also satisfies  (\ref{eq: abstract Pietsc mu2}).

   To prove  the remaining assertion,  assume that
 $\varphi$ is surjective. In this case the operator $\Phi$ in Lemma \ref{lem: rho1 rho2 probabilities} is an isometry (see \cite[Theorem 2.2]{Semadeni65}) and consequently,   $\Phi^*$ is a quotient operator.
Let  $\rho_2$ be a  probability measure on $K_2$ such that (\ref{eq: abstract Pietsc mu2}) holds.
   Since $\Phi^*$ is a surjective mapping,
   there exists $\rho\in  \mathcal{M}(K_1)$ such that
    $\Phi^*(\rho)=\rho_2.$ Being
   $\rho_2$  positive, by Lemma \ref{lem: rho1 rho2 probabilities} we will have that
    for  $h_i\in Z$
   and  $\lambda_i\geq 0$

\begin{multline*}
   \int_{K_1}\sum_{i=1}^k\lambda_i |h_i\circ \varphi|^pd\rho \leq \int_{K_1}( \sup_{k_1\in K_1} \sum_{i=1}^k\lambda_i |h_i\circ \varphi(k_1) |^p ) d\rho=\\
     \int_{K_2}( \sup_{k_1\in K_1} \sum_{i=1}^k\lambda_i |h_i\circ \varphi(k_1)|^p ) d\rho_2=\sup_{k_1\in K_1} \sum_{i=1}^k\lambda_i |h_i\circ \varphi(k_1)|^p.
\end{multline*}
The first equality above follows if we write   for   $c:=( \sup_{k_1\in K_1} \sum_{i=1}^k\lambda_i |h_i\circ \varphi(k_1)|^p )$
$$\int_{K_1}c\,d\rho=\int_{K_1}c\, (1_{K_2}\circ \varphi) d\rho=
\int_{K_2}c\cdot{1}_{K_2}d\rho_2=\int_{K_2}c\,d\rho_2.$$

Condition   (\ref{eq: abstract Pietsc mu2}) and equality  (\ref{eq: equal integrals}) applied to   $g(k_2)=|h_i(k_2)|^p$  ($i=1,\ldots, k$), along with the inequality above, imply

    \begin{multline}\label{eq:  Pietsch sums}
 \sum_{i=1}^k\lambda_i \|F(h_i)\|^p \leq  C^p \sum_{i=1}^k\lambda_i \int_{K_2} |h_i|^pd\rho_2
  =\\ C^p \sum_{i=1}^k\lambda_i \int_{K_1} |h_i\circ \varphi |^pd\rho
  \leq  C^p \sup_{k_1\in K_1} \sum_{i=1}^k\lambda_i |h_i\circ \varphi(k_1)|^p.
 \end{multline}

 Let $Q_1$ be  the cone consisting of all
  positive linear combinations of functions on $K_1$ of the form
  $$q_h(k_1):=\|F(h)\|^p- C^p|h(\varphi(k_1))|^p, \quad h\in Z. $$
  From (\ref{eq:  Pietsch sums})  we have that $Q_1$ can be separated  from the positive cone ${P_1:=\{f\in \mathcal{C}(K_1);\, f>0\}}$ by a linear functional $\widetilde{\rho}_1\in \mathcal{M}(K_1)$. Choosing $-\widetilde{\rho}_1$ if necessary, we know the existence of some  $c\in \mathbb{R}$ such that  for every $q\in Q_1, f\in P_1$
    $$\int_{K_1}q\,d{\widetilde{\rho}_1}
    \leq c <\int_{K_1}f\,d{\widetilde{\rho}_1}.  $$
  This inequality  used with   functions $\lambda q$  and  $f\equiv \mu$ where   $q$ is a fixed function in $Q_1$ and any $\lambda,\mu>0$, implies
   that   $c=0$. We get from this that   $\widetilde{\rho}_1$  is positive.  Normalizing  $\widetilde{\rho}_1$  we  get a  probability measure $\rho_1$ satisfying (\ref{eq: abstract Pietsc mu1}).

\end{proof}

\begin{proof}[Proof of Theorem \ref{thm: Lip=Lip}]Let  $ {\mathcal{L}\left(\Sigma_{\spaces{X}{1}{n}}\right)}$  be   the   space of $\Sigma$-operators associated to the  bounded multilinear  forms ${\mathcal{L}\left({\spaces{X}{1}{n}}\right)}$. They  are isometric spaces with the Lipschitz norm and the multilinear operator norm, respectively \cite[Proposition 3.3]{MFU}. Using   the isometric inclusion   $ \varphi :B_{\mathcal{L}\left(\Sigma_{\spaces{X}{1}{n}}\right)}\hookrightarrow B_{\Sigma_{\spaces{X}{1}{n}}^{\#}}$, we derive  that whenever $T$ satisfies (\ref{eq: def sigma p sumante}) for some $c>0$,  its associated mapping $f_T$  satisfies (\ref{eq: FarJ})  with $X:={\Sigma_{\spaces{X}{1}{n}}}$ and the same $c$. Thus,  $\pi_p^L(f_T)\leq \pi_p^{Lip}(T)$.

To  prove the reverse implication, we will consider  restrictions to finite dimensional subspaces.
For fixed     finite dimensional subspaces $E_i\subset X_i, \, i=1,\ldots, n$ let $\gamma$  be  the  reasonable cross-norm on $\otimes E_i$ induced by the inclusion in $\widehat{\otimes}_{\pi}X_i$.
Then $\Sigma_\gamma:=(\Sigma_{E_1,\ldots, E_n},\gamma)$ is a metric subspace  of $X$ and ${f_T}_{|\Sigma_{ \gamma}}$ is a Lipschitz mapping satisfying (\ref{eq: FarJ}). Since $\pi^L_p({f_T})$  is defined by means of evaluations on finite sets, it holds that
$$\pi^L_p(f_T)=\sup_{\{E_i\subset X_i\}_{i=1}^n}\{\pi^L_p({f_T}_{|\Sigma_{ \gamma}})\}.$$
On the  other hand,   as proved  in \cite[Example 3.2]{Samuel RACSAM 20}, the class of Lipschitz $p$-summing  multilinear operators is maximal. This means in particular that $$\pi^{Lip}_p(T)=\sup_{\{E_i\subset X_i\}_{i=1}^n}\{\pi^{Lip,\gamma}_p({T}_{|{E_1\times\cdots \times E_n}})\}. $$

 From now on, we will omit the spaces $E_i$ in the notation and will assume that $f_T$ is defined on $\Sigma_\gamma$.
By the Pietsch-type domination theorem proved in  \cite[Theorem 1]{FJ09}, there exists  a
probability  measure $\mu$ in $B_{\Sigma_\gamma^{\#}}$ such that for any
$x:=x_1\otimes\cdots\otimes x_n, y:=y_1\otimes\cdots\otimes y_n $ in  $\Sigma_{\gamma}$
\begin{equation}\label{eq: Pietsc FJ}
 \|f_T(x)-f_T(y)\|^p \leq \pi^L_p(f_T)^p \int_{B_{\Sigma_\gamma^{\#}}}|\zeta(x)-\zeta(y)|^pd\mu(\zeta).
 \end{equation}

First we  construct an induced probability measure $\mu_1$ on $B_{(\otimes_{\gamma}E_i )^{\#}}$ which   satisfies a Pietsch-type domination.
Consider the restriction mapping:

\[\begin{array}{ccc}
   \varphi:  B_{(\otimes_{\gamma}E_i)^{\#}} & \rightarrow & B_{\Sigma_{\gamma}^{\#}} \\
    \zeta  & \mapsto & \zeta_{|_{\Sigma_{\gamma}}}.
  \end{array}
\]
 Now we  identify $\Sigma_{\gamma}$ with its isometric copy   in $\mathcal{C}(B_{\Sigma_{\gamma}^{\#}})$
 by means of the natural isometric mapping $x\mapsto \delta_{x}$ where
 $\delta_x(\zeta):=\zeta(x)$ for each $\zeta\in B_{\Sigma_{\gamma}^{\#}}$.
 For each pair
$x:=x_1\otimes\cdots\otimes x_n, y:=y_1\otimes\cdots\otimes y_n $ in  $\Sigma_{\gamma}$, let   $h_{x,y}(\eta):=\eta(x)-\eta(y) \in \mathcal{C}(B_{\Sigma_{\gamma}^{\#}})$, $Z:=\{h_{x,y}; \,x,y\in \Sigma_{\gamma}\}$ and $F(h_{x,y}):=f_T(x)-f_T(y)$.
Inequality (\ref{eq: Pietsc FJ}) implies that  $F$ satisfies (\ref{eq: abstract Pietsc mu2}) in  Theorem \ref{thm: Induced Pietsch}.  Since McShane's extension of a Lipschitz mapping  guarantees that $\varphi$ is surjective,  $F$ satisfies also (\ref{eq: abstract Pietsc mu1}). That is,
 there is a probability  measure $\mu_1$ on  $B_{(\otimes_{\gamma}E_i)^{\#}}$ satisfying that
  for
 each pair
$x:=x_1\otimes\cdots\otimes x_n, y:=y_1\otimes\cdots\otimes y_n $ in  $\Sigma_{\gamma}$
\begin{equation}\label{eq: Pietsch FJ 2}
 \|f_T(x)-f_T(y)\|^p \leq \pi^L_p(f_T)^p  \int_{B_{(\otimes_{\gamma}E_i)^{\#}}}  |\zeta(x)-\zeta(y)|^pd{\mu}_1(\zeta).
 \end{equation}

  {{
  We are now in position to adapt the proof of the linear case \cite[Theorem 2]{FJ09} to our setting. There we can find the  justification for the following facts.  They can also  be tracked from  \cite[Proposition 6.41]{BenyLind}.

  Without loss of generality, we can assume that $\mu_1$ is separable. Let $\alpha:\widehat{\otimes}_{\gamma}  E_i \rightarrow  L_{\infty}(\mu_1)$ be the natural isometric embedding into $\mathcal{C}(B_{(\widehat{\otimes}_{\gamma} E_i)^{\#}})$ composed with the natural inclusion  into $L_{\infty}(\mu_1)$, $\alpha(w):=[\delta_w]$
 and let   $i_{\infty,p}$ be  the natural inclusion from  $L_{\infty}(\mu_1)$ into $L_{p}(\mu_1)$. The following properties hold:
     \begin{enumerate}
     \item The mapping $\alpha$ is weak$^*$ differentiable almost everywhere. This means that for  (Lebesgue) almost every $w_0\in \widehat{\otimes}_\gamma{{E_i}}$, there is a linear operator $D_{w_0}^{w^*}(\alpha):\widehat{\otimes}_\gamma{{E_i}}\rightarrow L_{\infty }(\mu_1)$ such that for all $f\in L_1(\mu_1)$ and for every $w\in \widehat{\otimes}_\gamma{{E_i}}$,
         $$\lim_{t\to 0}\bigg \langle \frac{\alpha(w_0+tw)-\alpha(w_0)}{t}, f\bigg \rangle=\langle D_{w_0}^{w^*}(\alpha)(w),f\rangle.  $$
     \item The operator $i_{\infty,p}\alpha$ is differentiable almost everywhere. This means that for (Lebesgue) almost every $w_0\in   \widehat{\otimes}_\gamma{{E_i}}$, there is a linear operator $D_{w_0}(i_{\infty,p}\alpha):\widehat{\otimes}_\gamma{{E_i}}\rightarrow L_{p }(\mu_1)$  such that
         $$\sup_{\|w\|\leq 1}\bigg \|  \frac{i_{\infty,p}\alpha(w_0+tw)-i_{\infty,p}\alpha(w_0)}{t}-D_{w_0}(i_{\infty,p}\alpha)(w) \bigg\|_{p}\to 0 \hspace{.2cm} {\mbox{as }}\; t\to 0.$$
          \end{enumerate}

   Select $w_0\in \widehat{\otimes}_\gamma{{E_i}}$ where both derivatives exist. Since $i_{\infty,p}$ is a is $w^*-w^*$ continuous mapping, equality
     $D_{w_0}(i_{\infty,p}\alpha)=i_{\infty,p}D_{w_0}^{w^*}(\alpha)$ holds. Let $\alpha_1(w):=\alpha(w+w_0)-\alpha(w_0)$. Then $D_{w_0}^{w^*}(\alpha)= D_{0}^{w^*}(\alpha_1)$,  $\alpha_1(0)=0$ and $\|D_{0}^{w^*}(\alpha_1)\|\leq Lip(\alpha)$.
    Consider the mapping   $
  \phi: B_{(\widehat{\otimes}_\gamma E_i)^{\#}}\rightarrow B_{(\widehat{\otimes}_\gamma E_i)^{\#}} $ defined as $\phi(\zeta)(w):=\zeta(w+w_0)-\zeta(w_0)$.  By the duality described in Lemma \ref{lem: rho1 rho2 probabilities} and being $\phi$  a bijective isometry,
  }}
  {{
     it induces the existence of a probability measure $\mu_2\in \mathcal{M}((\widehat{\otimes}_\gamma E_i)^{\#})$ such that for every $x,y \in \Sigma_{\gamma}$
   \begin{equation*}\label{eq: Pietsch FJ 3}
 \|f_T(x)-f_T(y)\|^p \leq \pi^L_p(f_T)^p  \int_{B_{(\otimes_{\gamma}E_i)^{\#}}}  |\zeta(x+w_0)-\zeta(y+w_0)|^pd{\mu}_2(\zeta).
 \end{equation*}}}

  {{
  {{ Let $\widetilde{Y}$ be the  finite dimensional subspace $\widehat{T}(\widehat{\otimes}_\gamma{{E_i}})$ of $Y$. For each $\epsilon >0$, let  $J:\widetilde{Y}\hookrightarrow \ell_{\infty}^m$ be a linear embedding with $\|J\|=1, \|J^{-1}\|\leq 1+\epsilon$. Let  $\alpha_2(w):=[\delta_{w+w_0}-\delta_{w_0}]\in L_{\infty}(\mu_2)$ for  $w\in \widehat{\otimes}_{\gamma} E_i$.
  The previous  inequality allows us to define ${\beta}( (i_{\infty,p}\alpha_{{2}})(x)) :=Jf_T(x)$ for every $x\in \Sigma_\gamma$, where $i_{\infty,p}$ denotes now the natural inclusion from $L_{\infty}(\mu_2)$ to $L_p(\mu_2)$. Then,  we have the  factorization $Jf_T={\beta} i_{\infty,p}  {{\alpha_{2}}}_{|_{\Sigma_{\gamma}}}$  with $Lip(\alpha_{{2}})\leq 1$, $Lip({\beta})\leq \pi^L_p(f_T)$. By the injectivity of   $\ell_{\infty}^m$ we can extend   ${\beta}$ to $L_p(\mu_2)$  preserving its norm to  obtain

\hspace{1cm}\xymatrix{
\Sigma_{\gamma} \ar[r]^ {{f}_T}\ar[d]^{{\alpha_{2}}_{|_{\Sigma_{\gamma}}}} &\widetilde{Y}\ar[r]^{J}& \ell_{\infty}^m  \\
L_{\infty}(\mu_2) \ar[rr]^ {i_{\infty,p}}&&  L_p(\mu_2) \ar[u]^{\widetilde{\beta}}. &
}

{{
The mapping
$\phi$, which determines $\mu_2$ as   the pull-back of  $\mu_1$,  induces the isometric onto isomorphisms
 $A_p:L_p(B_{(\otimes_{\gamma}E_i)^{\#}},\mu_2)\rightarrow L_p(B_{(\otimes_{\gamma}E_i)^{\#}},\mu_1)$ defined as $A_p(f):=f\circ \phi $, for $1\leq p \leq \infty$.   If $1\leq p < \infty$ and $1=\frac{1}{p}+ \frac{1}{p'}$ then $A_{p'}^{-1}=A_p^*. $  These relations along with the fact that $A_1^*$ is $w^*-w^*$ continuous, guarantee that the following derivatives of $\alpha_2$ at $0$ exist and satisfy $D_0^{w^*}(\alpha_2)=A_1^* D_0^{w^*}(\alpha_1)$ and $i_{\infty,p} D_0^{w^*}(\alpha_2)=D_0(i_{\infty,p}\alpha_2)$. Consequently, $\|D_{0}^{w^*}(\alpha_2)\|\leq Lip(\alpha_2)\leq 1$.
 }}

}}

The following calculations  are  the same that the ones in \cite[Theorem 2]{FJ09}, evaluated only at vectors in $\Sigma_\gamma$.

 Let $\widetilde{\beta}_n(w):=n\widetilde{\beta}(\frac{w}{n})$.  Then   $Lip(\widetilde{\beta}_n)=  Lip(\widetilde{\beta})$. Using that $Jf_T$ is homogeneous, for each $x_1\otimes \cdots \otimes x_n \in \Sigma_{\gamma}$ it holds that
   \begin{multline*}
 \|J\circ f_T(x_1\otimes \cdots \otimes x_n)-\widetilde{\beta}_ni_{\infty,p}D_0^{w^*}\alpha_{{2}}(x_1\otimes \cdots \otimes x_n) \|=\\\|\widetilde{\beta}_nni_{\infty,p}\alpha_{{2}}(\frac{x_1\otimes \cdots \otimes x_n}{n})-\widetilde{\beta}_ni_{\infty,p}D_0^{w^*}\alpha_{{2}}(x_1\otimes \cdots \otimes x_n)  \|   \leq \\
    Lip(\beta)\|ni_{\infty,p}\alpha_2(\frac{x_1\otimes \cdots \otimes x_n}{n})-D_0(i_{\infty,p}\alpha_2({x_1\otimes \cdots \otimes x_n}))\|\xrightarrow[n \to \infty]{} 0.
 \end{multline*}

 Since   $\{\widetilde{\beta}_n\}$ is a bounded set in   $Lip_0(L_p(\mu_2), \ell_{\infty}^m)$, it  has   a cluster point $\beta_0$ in it.    Let $f_D:=D_0^{w^*}({\alpha_{{2}}})_{|_{\Sigma_\gamma}}$ and $T_D$  be  the $\Sigma$-operator and the multilinear mapping  respectively,
 associated to $D_0^{w^*}({\alpha_2})$.
  Then we have   the factorization $Jf_T={\beta_0}i_{\infty,p} f_{D}$ with $Lip(f_{{D}})\leq \|D_0^{w^*}(\alpha_{{2}})\|\leq Lip(\alpha_2)$ and the   factorization of the multilinear mapping  $JT=\beta_0 i_{\infty,p} T_{D}$ with $\|T_{D}\|=\|{D_0^{w^*}({\alpha_{{2}}})}\|_{{\widehat{\otimes}_\pi E_i\rightarrow L_{\infty}}} \leq \|{D_0^{w^*}({\alpha_{{2}}})}\|_{{\widehat{\otimes}_\gamma E_i\rightarrow L_{\infty}}}\leq Lip(\alpha_2)$.  The result already follows from these computations, if we use  an  observation in   \cite[Section 5]{Ang-FU}. For the sake of completeness we write the argument. Let $k\in\mathbb{N}$, $i=1,\ldots,k$ and  $x_i:=x^i_1\otimes\cdots\otimes x^i_n, y_i:=y^i_1\otimes\cdots\otimes y^i_n\in \Sigma_\gamma$. Using that  the
$p$-summing norm  $\Pi_p(i_{\infty,p})$  of the linear inclusion is one,  we have that for
 $x_i, y_i \in \Sigma_{\gamma}$
 \begin{multline*}
\sum_{i=1}^k \left\|Jf_T(x_i)-Jf_T(y_i)\right\|^p =
  \sum_{i=1}^k \left\|\beta_0i_{\infty,p}f_{D}(x_i)-\beta_0i_{\infty,p}f_{D}(y_i)\right\|^p
  \\\leq {Lip(\beta_0)}^p
  {\Pi_p(i_{\infty,p})}^p \sup_{ \varphi\in B_{L_{\infty}^*}} \sum_{i=1}^k \left|\varphi {D_0^{w^*}(\alpha_2)}(x_i)-\varphi D_0^{w^*}(\alpha_2)(y_i)\right|^p
\\ \leq    Lip(\beta_0)^p \|D_0^{w^*}(\alpha_2)\|^p \sup_{ \zeta \in  B_{(\widehat{\otimes}_{\gamma}E_i)^*}} \sum_{i=1}^k \left|\zeta(x_i)-\zeta(y_i)\right|^p \\ \leq Lip(\beta_0)^p Lip(\alpha_2)^p \sup_{ \varphi \in B_{(\widehat{\otimes}_{\pi}X_i)^*}} \sum_{i=1}^k \left|{{\varphi}}\left(x_i\right)-{\varphi}\left(y_i\right)\right|^p.
\end{multline*}
 The last inequality holds because $\widehat{\otimes}_{\gamma}E_i$ is a closed subspace of $\widehat{\otimes}_{\pi}X_i$ and consequently each  norm-one linear form defined on $\widehat{\otimes}_\gamma E_i$ has a norm-one extension. Then $\pi_p^{Lip}({{f_T}_{|_{\Sigma_\gamma}}})\leq (1+\epsilon) Lip(\alpha_2) Lip({\beta_0}) $. Since this is true for arbitrary $\epsilon$ and also for arbitrary  finite dimensional spaces $E_i\subset X_i$, $i=1,\ldots,n$, we have that   $\pi_p^{Lip}({f_T})\leq  Lip(\alpha_2) Lip( {{\beta}_0}) $ and, consequenlty,  $\pi_p^{Lip}(T)\leq \pi_p^L(f_T)$.

 }}

\end{proof}

This result answers affirmatively  Question 7.1 in \cite{Ang-FU}.
 In that  paper  it was also introduced the Lipschitz $p$-summability of a $\Sigma$-operator. Theorem \ref{thm: Lip=Lip} clearly implies that for $\Sigma$-operators  both  notions  coincide.

\vspace{.5cm}

\vspace{.5cm}

\section{ Comparison of Lipschitz $p$-summabilities in the case of  other reasonable crossnorms}\label{sec: beta Sigma op}

Here we consider the analgous question when a  reasonable cross-norm $\gamma$,  other  than the projective norm, is  defined on $X_1\otimes\cdots\otimes X_n$. Namely, if  it is true that a multilinear mapping $T\in {\mathcal{L}_{\gamma}\left(\spaces{X}{1}{n}; Y \right)}$ is $\gamma$-Lipschitz $p$-summing (\ref{eq: def sigma p sumante}) if and only if  its associated  $\Sigma$-operator $f_T:(\Sigma_{X_1,\ldots,X_n}, \gamma)\rightarrow Y $  is a  Lipschitz $p$-summing mapping (\ref{eq: FarJ}).

 Using    the isometric inclusion   $ \varphi :B_{\mathcal{L}_\gamma\left(\Sigma_{\spaces{X}{1}{n}}\right)}\hookrightarrow B_{\Sigma_{\gamma}^{\#}}$ it is  direct to prove that if $T$ is  a $\gamma$-Lipschitz $p$-summing multilinear operator, then  $f_T$ on $\Sigma_\gamma$ is   a Lipschitz $p$-summing mapping and $\pi_p^{L}(f_T:\Sigma_\gamma\rightarrow Y)\leq \Pi_p^{Lip,\gamma}(T)$.  We will see that the reciprocal statement does not always hold. Note that in this case we are   assuming that $\widehat{T}$ is continuous on $\widehat{\otimes}_{\gamma}X_i$.

\begin{lemma}\label{lem:  the same for Sigmas}
Let $ \gamma$ be a reasonable cross-norm defined on $X_1\otimes\cdots\otimes X_n$.
\begin{enumerate}
  \item If  $\Sigma_{\gamma}:=(\Sigma_{X_1,\ldots, X_n},\gamma)$,  then the identity mappings $Id:\Sigma\rightarrow \Sigma_{\gamma}$ and   $Id: B_{\Sigma_{\gamma}^{\#}}\rightarrow  B_{\Sigma^{\#}}$  are bi-Lipschitz  with $Lip(Id)=1$ and $Lip(Id^{-1})  \leq 4^{n-1}$.
  \item For  $T\in {\mathcal{L}_{\gamma}\left(\spaces{X}{1}{n}; Y \right)}$ and $1\leq p < \infty$, $ f_T:\Sigma_\gamma\rightarrow Y $ is a Lipschitz $p$-summing mapping if and only if $f_T:\Sigma \rightarrow Y $ is a Lipschitz $p$-summing mapping. In this case  $\pi_p^{L}(f_T)\leq 4^{n-1} \pi_p^{L}(f_T:\Sigma_{{\gamma}}\rightarrow Y)$.  \end{enumerate}

\end{lemma}
\begin{proof}

Recall that  $f_T(x_1\otimes\cdots \otimes x_n):=T(x_1,\ldots, x_n)$.
The first assertion in (1) is proved in  \cite[Theorem 2.1]{MFU} and the second follows immediately from it.
To prove (2) it is enough to observe that
  inequality (\ref{eq: FarJ}) holds for the supremum on  $B_{\Sigma^{\#}}$ if and only if it holds on $B_{\Sigma_{\gamma}^{\#}}$   with a constant at most $4^{n-1}C$.
\end{proof}

 \begin{example} \label{ex: not Lip=Lip Hilberttensorproduct}
   Let $H$ be the    Hilbert  reasonable cross-norm on $\ell_2\otimes \ell_2$ (see,  e.g.  \cite[Definition 5.8]{Matos03}) and consider $T:\ell_2\times \ell_2\rightarrow \ell_2\widehat{\otimes}_H \ell_2$ defined as $T((a_i)_i,(b_j)_j)=\sum_{i=1}^{\infty}a_ib_ie_i\otimes e_i$.
   Then, for every $1\leq p <\infty$, $f_T$ is Lipschitz $p$-summing as a Lipschitz mapping  (\ref{eq:  FarJ}) on $\Sigma_{H}$,  $T\in  \mathcal{L}_{H }(\ell_2, \ell_2;  \ell_2\widehat{\otimes}_H \ell_2) $, but $T$ is not a $H$-Lipschitz $p$-summing bilinear mapping.

   \

  To prove the assertions, recall  first that the completed space $\ell_2\widehat{\otimes}_H \ell_2$ is a Hilbert space and  $\{e_i\otimes e_j\}_{i,j}$ is an ortonormal basis for it.
The  linear mapping   associated with $T$ satisfies
 $\widehat{T}\in \mathcal{L}(\ell_2\widehat{\otimes}_H \ell_2, \ell_2\widehat{\otimes}_H \ell_2)$ which   says that   $T\in  \mathcal{L}_{H }(\ell_2, \ell_2;  \ell_2\widehat{\otimes}_H \ell_2) $.

 $T$ can be factorized as $T=i\circ S$   where  $S\in \mathcal{L}(\ell_2, \ell_2;  \ell_1)$ is defined as $S((a_n)_n,(b_j)_j)=(a_jb_j)_j$ and $i:\ell_1\rightarrow \ell_2\hookrightarrow  \ell_2\widehat{\otimes}_H \ell_2$ is the natural inclusion. Since  $i$  is absolutely summing, $T$ is a Lipschitz $1$-summing bilinear operator. By Theorem \ref{thm: Lip=Lip}, $f_T:\Sigma\rightarrow \ell_2\widehat{\otimes}_H\ell_2$ is a Lipschitz $1$-summing mapping.  By Lemma \ref{lem:  the same for Sigmas}, we also have that $f_T:\Sigma_{{H}}\rightarrow \ell_2\widehat{\otimes}_H\ell_2$ is Lipschitz $1$-summing.

   Now we check that $T$ is not a $H$-Lipschitz $1$-summing bilinear operator. By \cite[Theorem 5.2]{Ang-FU}, this is equivalent to prove that ${T}$ is not a Hilbert-Schmidt bilinear operator. Thus, it is equivalent to prove that $\widehat{T}\in\mathcal{L}(\ell_2\widehat{\otimes}_H\ell_2,\ell_2\widehat{\otimes}_H\ell_2)$ is not a Hilbert-Schmidt linear operator \cite[Proposition 5.10]{Matos03}.    But this is  clear since $\{e_i\otimes e_j\}_{i,j}$ is an orthonormal basis of    the space and $\sum_{i,j=1}^{\infty}\|\widehat{T}(e_i\otimes e_j)\|^2$  is not finite.
  The same example serves to prove the case for any  $1<p<\infty$.

\end{example}

\section*{Acknowledgement}

The author wishes to thank Samuel Garc\'ia-Hern\'andez for helpful discussions during the preparation of this manuscript  and the  anonymous referee  whose suggestions helped  improve and clarify it.

\bibliographystyle{amsplain}

\begin{thebibliography}{99}

\bibitem{Ang-FU}
Angulo-L\'opez J.C.;  Fern\'andez-Unzueta M. Lipschitz p-summing multilinear operators. J. Funct. Anal. 279 (2020), no. 4.
doi.org/10.1016/j.jfa.2020.108572




 \bibitem{BenyLind} Benyamini, Y;  Lindenstrauss, J. \emph{Geometric Nonlinear Functional Analysis Volume 1}; American Mathematical Society Colloquium Publications Volume 48, 2000.


 \bibitem{BotPelRue2} Botelho G.;   Pellegrino D.;  Rueda P. { A unified Pietsch domination theorem. J. Math. Anal. Appl}. 365 (2010), no. 1, 269-276.





\bibitem{DeFlo} Defant, A.; Floret, K. Tensor norms and operator ideals. \textit{North-Holland Mathematics Studies}, 176, Amsterdam, 1993.



         \bibitem{DJT} Diestel,  Jarchow H.;  Tonge A. \emph{Absolutely Summing Operators. Cambridge Univ. Press}, 1995.


\bibitem{FJ09}  Farmer, J.D.; Johnson, W. B.
Lipschitz p-summing operators.
Proc. Amer. Math. Soc. 137 (2009), no. 9, 2989-2995.



\bibitem{MFU} Fern\'andez-Unzueta, M. The Segre cone of Banach spaces and multilinear mappings. Linear Multilinear Algebra 68 (2020), no. 3, 575-593.



  \bibitem{Matos03}  Matos. M.C. {Fully absolutely summing mappings}, Math. Nachr. \textbf{258} (2003), 71-89.

\bibitem{Ryan-libro}  Ryan, R.  Introduction to tensor products of Banach spaces. Springer Monographs in Mathematics. \textit{Springer-Verlag London, Ltd., London}, 2002.

\bibitem{Samuel RACSAM 20}  Garc\'\i a-Hern\'andez, S.,  The duality between ideals of multilinear operators and tensor norms. Rev. R. Acad. Cienc. Exactas Fís. Nat. Ser. A Mat. RACSAM 114 (2020), no. 2
\bibitem{Semadeni65} Semadeni, Z.
Spaces of continuous functions on compact sets.
Advances in Math. 1 (1965), fasc. 3, 319-382.

\bibitem{Weaver-libro} Weaver, N.  Lipschitz algebras. Second edition. World Sci.  2018.

 \end{thebibliography}

\end{document}